\documentclass[11pt,twoside]{preprint}
\usepackage[full]{textcomp}
\usepackage[osf]{newtxtext}
\usepackage{amsmath,amsthm,amssymb,enumerate}
\usepackage{hyperref}
\usepackage{mathrsfs}
\usepackage{breakurl}
\usepackage{appendix}
\usepackage{color}
\usepackage{tikz-cd}

\definecolor{DGreen}{rgb}{0,0.55,0}
\usepackage{xcolor}
\usepackage{mhequ}
\usepackage{comment}

\usepackage{microtype}
\usepackage[a4paper,innermargin=1.4in,outermargin=1.4in,bottom=1.5in,marginparwidth=1.5in,marginparsep=3mm]{geometry}

\renewcommand{\theequation}{\thesection.\arabic{equation}}

\newtheorem{theorem}{Theorem}[section]
\newtheorem{lemma}[theorem]{Lemma}
\newtheorem{proposition}[theorem]{Proposition}
\newtheorem{corollary}[theorem]{Corollary}

\theoremstyle{definition}
\newtheorem{assumption}[theorem]{Assumption}
\newtheorem{definition}[theorem]{Definition}

\theoremstyle{remark}
\newtheorem{remark}[theorem]{Remark}

\def\scal#1{\langle #1 \rangle}

\allowdisplaybreaks

\newcommand\Z{\mathbb{Z}}
\newcommand{\N}{\mathbb{N}}

\def\one{\mathbf{1}}

\def\eps{\varepsilon}

\newcommand{\R}{\mathbb{R}}
\newcommand{\E}{\mathbb{E}}
\newcommand{\F}{\mathbb{F}}
\def\cF{\mathcal{F}}

\renewcommand{\P}{\mathbb{P}}
\def\D{(-\Delta)}

\def\d{\partial}

\def\cX{\mathcal{X}}
\def\cH{\mathcal{H}}

\begin{document}
\title{Porous media equations with multiplicative space-time white noise}
\author{K. Dareiotis, M. Gerencs\'er, B. Gess}
\maketitle

\begin{abstract}
The existence of martingale solutions for stochastic porous media equations driven by nonlinear multiplicative space-time white noise is established in spatial dimension one. The Stroock-Varopoulos inequality is identified as a key tool in the derivation of the corresponding estimates. 
\end{abstract}


\section{Introduction}
 We  consider the equation
\begin{equ}\label{eq:main}
\partial_t u=\d_x^2 ( u^{[m]}) +\sigma(x,u)\xi,
\end{equ}
on the space-time domain $Q:=[0,T]\times I:=[0,T]\times(0, 1)$, with homogeneous Dirichlet boundary conditions, and with some initial condition $u^{(0)}$. The noise $\xi$ is white in both space and time, $u^{[m]}:=|u|^{m-1}u$, $m\in(1,\infty)$, and $T\in(0,\infty)$.

While stochastic porous media equations attracted significant attention,
all available results  concerning multiplicative noise  pose strong spatial colouring conditions on the noise.
Indeed, the monotone operator approach, for which we refer the reader to the monographs \cite{BDPR,Liu}, requires noise with Cameron-Martin space $L^2([0,T],\cH)$, where $\cH=H^{3/2+}$. 
Another notable restriction of this approach is that when the mapping $u\mapsto\sigma(x,u)$ is given by a pointwise composition, only affine linear   diffusion  coefficients are covered. In the case of nonlinear diffusion coefficients, the existence of martingale solutions was shown in \cite{RocknerGoldys}, again restricting to spatially colored noise with $\cH=H^{1/2+}$.  Recent development of an $L^1$-based theory, see for example \cite{BVW15,GS16-2,Debussche, GH18,FG18,DGG19} and references therein, has lead to the (pobabilistically) strong existence and uniqueness for a large class of nonlinear diffusions $\sigma$ with spatially colored noise. The most lenient conditions are from \cite{DGG19}, which corresponds to $\sigma \in C^{1/2+}$ and $\cH=H^{1/2+1/(m\vee 2)+}$. Needless to say,  all of these results are quite far from the space-time white noise case $\cH=L^2$.

Equation \eqref{eq:main} can also be seen as an example of a singular SPDE. The theory of these equations have seen major advances recently thanks to the theories of \cite{H0, GIP}.
Quasilinear singular SPDEs, first studied by \cite{OttoWeber}, have recently been also solved with space-time white noise via these theories \cite{GerencserHairer, GERENCSER2020, bailleul2019}. However, the degeneracy of the leading order operator prevents any of these works to apply for the study of (1.1).
The additional (It\^o-) structure of the equation,  however,  allows for stochastic analytic tools.

In the main result of this work, Theorem \ref{thm:weak-existence} below,  we establish the probabilistically weak existence of solutions for a class of diffusion nonlinearities.
The scope of Theorem \ref{thm:weak-existence} is quite large: any continuous $\sigma$ satisfying a mild growth condition fits in the framework. This in particular covers  the case $\sigma(r)= \sqrt{r}$ which is known to be relevant in scaling limits of interacting branching particle systems (see Section \ref{sec:example} below).  

In order to prove  the existence of solutions, we obtain estimates for solutions of  viscous approximations of \eqref{eq:main} with finitely many modes of noise, the well-posedness of which is guaranteed by \cite[Theorem 3.1]{DGT19}. These estimates should be in spaces of positive regularity in order to guarantee compactness in some $L^p$ space (in space time). At the same time the regularity exponent should be relatively small in order to avoid blow ups appearing due to the irregularity of the noise.  We identify the Stroock-Varopoulos inequality as a key ingredient in obtaining  such estimates that are compatible with the non-linear nature of \eqref{eq:main}. It is remarkable that this inequality, which originates in the analysis of non-local porous media equations, proves to be vital to the local but irregular setting of \eqref{eq:main}.


In such generality, no strong uniqueness is expected to hold for \eqref{eq:main}, since it is not even true in the semilinear case $m=1$, see \cite{Mueller-M-P}.
It is,  however,  reasonable to expect strong uniqueness when $\sigma$ is, say, Lipschitz continuous, which remains an open question. In this article we show that strong existence and uniqueness hold when, roughly speaking, $\sigma(x,u)$ behaves like $ u^{[(m+1)/2]}$ around the origin (see Proposition \ref{thm:strong-wellposedness} below). 

The rest of the article is structured as follows. In Section \ref{sec:example} we give an example on how a heuristic scaling of a simple system of interacting particles gives rise to an SPDE of the form \eqref{eq:main}.
In Section \ref{sec: formulation} we state the main results. The proof of Theorem \ref{thm:weak-existence} is divided into a priori estimates for some approximating equations in Section \ref{sec: a priori} and the passage to the limit in Section \ref{sec:limiting}. 
The proof of Proposition \ref{thm:strong-wellposedness} is given in Section \ref{sec:strong H1}.

\subsection{A heuristic derivation}\label{sec:example}
In $\R$, let us consider   particles $(X^i_t)_{t \geq 0}$ (of negligible mass),  for $i \in I_t$,  interacting through a potential $V$.  Here,  $I_t$ is an index set depending on time with $|I_0| \sim N$, 
The system undergoes critical branching: each particle, with rate one,  dies and leaves behind offspring with the expected number of offspring being one.
During their lifetime the particles $X^i_t$ evolve under the dynamics
$$                                    
d X^i_t = \sum_{j \in I_t}  \d_x V (X^i_t-X^j_t) \, dt,  
$$
where $V :  \R \to \R$ is a compactly supported, non-negative function integrating to one. After introducing the rescaling $Y^{i,N}_t= N^{-2/3} X^i_{Nt}$, one has that
\begin{equs}
dY^{i,N}_t = \frac{1}{N}  \sum_{j \in \tilde{I}_{t}} \d_x V_{N^{-2/3}} (Y^{i,N}_t-Y^{j,N}_t),
\end{equs}
where $\tilde I_t = I_{Nt}$, $V_\alpha(\cdot) := \alpha^{-1} V(\alpha^{-1} \cdot)$ for $\alpha >0$, and the particles $Y^i_t$ branch with rate $N$. The ultimate goal is to let $N \to \infty$. Since $N^{2/3} \ll N$ we make the following simplification: we consider the system with the same branching mechanism but dynamics given by 
\begin{equs}
dY^{i,N, \eps}_t = \frac{1}{N}  \sum_{j \in \tilde{I}_{t}} \d_x V_\eps (Y^{i,N, \eps}_t-Y^{j,N, \eps}_t). 
\end{equs}
Let us denote by $\mu^{N, \eps}_t$ the empirical measure of the above system at time $t$, that is, 
\begin{equs}
\mu^{N, \eps}_t = \frac{1}{N} \sum_{i \in \tilde{I}_t} \delta_{Y^{i, N, \eps}_t}.
\end{equs}
It follows from \cite[Theorem 1]{Meleard} that for $N \uparrow \infty$,  $(\mu^{N,\eps}_t)_{t \in [0,T]}$ converges --in an appropriate sense-- to a non-negative measure valued stochastic process $(\mu^\eps_t)_{t \in [0,T]}$ which satisfies
\begin{equs}
\partial_t \mu^\eps = \d_x\left( \mu^\eps  \left( \d_x V_\eps * \mu^\eps \right)  \right) + c \sqrt{\mu^\eps} \xi,
\end{equs}
where $c$ is a constant depending on the variance of the branching mechanism and $\xi$ is space time white noise. Informally, since $V_\eps$ tends to $\delta_0$,  passing to the limit $\eps\to 0$ in the above equation leads to
$$
\d_t \mu  = \frac{1}{2} \d^2_x ( \mu^2)  + c\sqrt{\mu} \xi. 
$$ 
In the deterministic case, that is if $c=0$, the limit $\eps \to 0$ has been rigorously justified in \cite{LMG01}.

\subsection{Notation}
Due to the low regularity we always work with
weak (in the terminology of e.g. \cite{Bonforte-Vazquez}, `weak dual') solutions in the PDE sense, 
and consider both strong and weak solutions in the probabilistic sense. 
For the former, fix the probability space $(\Omega,\cF,\P)$
on which the space-time white noise $\xi$ is given.
Recall that this means a collection of jointly Gaussian centred random variables
$\xi(\varphi)$,  $\varphi\in L^2(Q)$,
with covariance $\E\big(\xi(\varphi)\xi(\bar\varphi)\big)=(\varphi,\bar\varphi)_{L^2(Q)}$.
For the remainder of the article we set $e^k(x)= \sqrt{2}\sin (\pi k x)$ for $k \in \N=\{1,2,..\}$. We have that  $(e^k)_{k\in\N}$ is an orthonormal basis of $L^2(I)$.
For each $k\in\N$, set $(w^k_t)_{t\in[0,T]}$ to be a continuous modification of the collection of random variables $(\xi(\one_{[0,t]}e^k))_{t\in[0,T]}$.
It is well-known that such modifications exist, as is the fact that $w^1,w^2,\ldots$ is a sequence of independent Wiener processes.
We denote by $\F=(\cF_t)_{t\in[0,T]}$ the right continuous completion of the filtration generated by them.

Function spaces in the spatial variable $x\in I$ are denoted by the lower index $x$. For notational simplicity we do not make the $x$-dependency of $\sigma(x,u)$ explicit when convenient. The set $(e^k)_{k \in \N}$
consists of the eigenvectors of the $\D$ with Dirichlet boundary condition on $\partial I$, with corresponding eigenvalues $\lambda_k=(\pi k)^2$. For $\gamma\geq 0$ we introduce the space
$$
H^\gamma_x = \left\{ v  \in L^2_x : \sum_{k\in\N} \lambda^\gamma_k |(v, e^k)_{L^2_x}|^2 < \infty \right\},
$$
endowed with the norm 
$$
\|v\|^2_{H^\gamma_x}:= \sum_{k\in\N} \lambda^\gamma_k |(v, e^k)_{L^2_x}|^2.
$$
Here and in the sequel if $H$ is a Hilbert space, $(\cdot,\cdot)_H$ stands for the  inner product in $H$.
We define $H^{-\gamma}_x$ to be the dual of $H^\gamma_x$ in the Gelfand triple 
$H^\gamma_x\subset L^2_x\equiv(L^2_x)^*\subset(H^\gamma_x)^*$.
Extending $(\cdot,\cdot)_{L^2_x}$ to the $H^\gamma_x-H^{-\gamma}_x$ duality denoted by $\langle \cdot, \cdot \rangle$, the norm of an element $v^*\in H^{-\gamma}_x$ is given by
\begin{equ}
\|v^*\|_{H^{-\gamma}_x}^2=\sum_{k\in\N}\lambda_k^{-\gamma}|\langle v^*,e^k\rangle|^2.
\end{equ}
It is easy to see that $L^2_x$ is dense in $H^{-\gamma}_x$, and therefore so is $C^\infty_c(I)$.
It is also easy to verify that for $\gamma>0$ the embedding $L^2_x\subset H^{-\gamma}_x$ is compact.
For $\beta \in \R $ we define
\begin{equs}
\D^{\beta/2} \phi:= \sum_{k\in\N} \lambda_k ^{\beta/2} (\phi, e^k)_{L^2_x} e^k, \qquad \text{for $\phi \in C^\infty_c(I)$}.
\end{equs}
For any $\gamma  \in \R$
the operator $\D^{\beta/2}$ extends to an isometry 
$$
\D^{\beta/2}: H^{\gamma}_x \to H^{\gamma-\beta }_x.
$$
It follows that for $\gamma_1> \gamma_2$ the embedding $H^{\gamma_1}_x \subset H^{\gamma_2}_x$ is compact. 
For all $\gamma\in \R$, $H^\gamma_x$ is a Hilbert space. Using the inner product of $H^{-1}_x$ to identify it with its own dual, one can consider
the Gelfand triple $L^{m+1}_x\subset H^{-1}_x\equiv (H^{-1}_x)^*\subset (L^{m+1}_x)^*$.  
The operator $u\mapsto \Delta u^{[m]}$ maps $L^{m+1}_x$ to $(L^{m+1}_x)^*$ and the action of $\Delta u^{[m]}$ on an element $ \phi \in L^{m+1}_x$ is given by 
$$
{}_{(L^{m+1}_x)^*} \langle \Delta u^{[m]}, \phi \rangle_{L^{m+1}_x}= -( u^{[m]}, \phi)_{L^2_x}. 
$$
For more details we refer to \cite[Ex.~4.1.11]{Rock}.

Function spaces in the temporal variable $t$, whenever given on the whole time horizon $[0,T]$, are denoted by the lower index $t$. For instance, $L^2_tH^\gamma_x$ stands for $L^2([0,T],H^\gamma_x)$. Occasionally the time horizon will be different, in these cases we specify the domains. 
In the temporal and spatial variable we will also consider the Sobolev-Slobodeckij spaces
$W^{\gamma,p}$ (see e.g. \cite[Sec.~4.2]{Triebel}). Their relevant properties are stated in Proposition \ref{prop:Sobolev} below.
By $\dot W^{\gamma,p}_{x}$ we denote the closure of $C^\infty_c(I)$ in $W^{\gamma,p}_x$.
Finally,  spaces of functions on $\Omega$ (which will always be $L^p$ spaces)  are denoted by the lower index $\omega$. When $L^p$ spaces are considered on $Q$ or $\Omega\times Q$, we write $L^p_{t,x}$ or $L^p_{\omega,t,x}$.

\begin{proposition}\label{prop:Sobolev}
\begin{enumerate}[(i)]
\item\label{p:slobodeckij} \cite[Rmk~4.4.2/2]{Triebel}. Let $p\in(1,\infty)$, $\gamma\in(0,1)$. Then an equivalent norm in $W^{\gamma,p}_x$ is given by
\begin{equ}\label{eq:norm} 
\Big( \|v\|_{L^p_x}^p+\int_{I\times I} \frac{|v(x)-v(y)|^p}{|x-y|^{1+\gamma p} }\, dx dy\Big)^{1/p};
\end{equ}
\item\label{p:dot} \cite[Thm~4.3.2/1]{Triebel}. Let $p\in(1,\infty)$, $\gamma\in(-\infty,1/p]$. Then $\dot W^{\gamma,p}_x=W^{\gamma,p}_x$;
\item\label{p:spectral} 
\cite[Eq~2.11]{Bonforte-Vazquez}. Let $\gamma\in(0,1)\setminus\{1/2\}$. Then $\dot W^{\gamma,2}_x=H^\gamma_x$;
\item\label{p:dual} \cite[Thm~4.8.2]{Triebel}. Let $p\in(1,\infty)$, $\gamma\in[0,\infty)$ such that $\gamma-1/p\notin\Z$. Then the dual of $\dot W^{\gamma,p}_x$, viewed as a subset of distributions, is $W^{-\gamma,p'}_x$, where $1/p+1/p'=1$;
\item\label{p:interpolation} \cite[Thm~4.3.1/1]{Triebel}, \cite[Eq~2.4.1/(8)]{Triebel}. Let $-\infty<\gamma_0<\gamma_1<\infty$, $1<p_0,p_1<\infty$, $\theta\in(0,1)$, and define
\begin{equ}
\gamma=(1-\theta)\gamma_0+\theta\gamma_1,\qquad p=\big((1-\theta)p_0^{-1}+\theta p_1^{-1}\big)^{-1}.
\end{equ}
Suppose $\gamma_0,\gamma_1,\gamma\notin\N$.
Then
one has
\begin{equ}\label{eq:interpolation}
\|v\|_{W^{\gamma,p}_x}\leq\|v\|_{W^{\gamma_0,p_0}_x}^{1-\theta}\|v\|_{W^{\gamma_1,p_1}_x}^\theta.
\end{equ}
\end{enumerate}
\end{proposition}
\begin{remark}\label{rem:trivial}
Interpolation between $H^\gamma_x$ spaces is straightforward from the definition.
\end{remark}

\section{Formulation and main results} \label{sec: formulation}
\begin{definition}					\label{def:strong}
A \emph{strong solution} of \eqref{eq:main} is an $H^{-1}_x$-valued continuous $\F$-adapted process $u$,
that furthermore belongs to $L^{m+1}_{t,x}$ almost surely and such that  for all $\phi \in L^{m+1}_x$  almost surely  the equality 
\begin{equ}
(u(t), 
\phi)_{H^{-1}_x}=(u^{(0)},\phi)_{H^{-1}_x} -\int_0^t(u^{[m]}(s), \phi)_{L^2_x}\,ds+\sum_{k\in\N}\int_0^t\big( \sigma(u(s))e^k, \phi\big)_{H^{-1}_x}\,dw^k_s
\end{equ}
holds  for all $t\in[0,T]$. 
\end{definition}
\begin{definition}					\label{def:weak}
A \emph{weak solution} of \eqref{eq:main} is a collection $\{(\bar\Omega,\bar \cF,\bar \P), \bar \F, (\bar w^k)_{k\in\N}, \bar u \}$, such that  $(\bar\Omega,\bar \cF,\bar \P)$ is a probability space,   $\bar \F=(\bar\cF_t)_{t\in[0,T]}$ is
  a complete filtration of $\bar \cF$, $(\bar w^k)_{k\in\N}$ is  a sequence of independent $\bar\F$-Wiener processes,
and $\bar u$ is
an $H^{-1}_x$-valued continuous $\bar\F$-adapted process,
that furthermore belongs to $L^{m+1}_{t,x}$ almost surely and such that  for all $\phi \in L^{m+1}_x$  almost surely  the equality
\begin{equ}
(\bar u(t), 
\phi)_{H^{-1}_x}=(\bar u^{(0)},\phi)_{H^{-1}_x} -\int_0^t(\bar u^{[m]}(s), \phi)_{L^2_x}\,ds+\sum_{k\in\N}\int_0^t( \sigma(\bar u(s))e^k, \phi)_{H^{-1}_x}\,dw^k_s
\end{equ}
holds  for all $t\in[0,T]$, where $\bar u^{(0)}\overset{d}{=}u^{(0)}$.
\end{definition}

For the definition to be meaningful, some assumption of $\sigma$ has to be imposed.
It is a consequence of Lemma \ref{lem:krylov lemma} below that under Assumption \ref{as:sigma}, for all $t\in[0,T]$ the series of stochastic integrals converge in probability.

\begin{remark}[On the notion of solution]
 Notice that $(-\Delta)^{-1} (-\Delta) \psi =  \psi$ whenever $\psi \in H^\gamma_x$, for any $\gamma\in\R$.
 It follows that by choosing $\phi = - \Delta \psi $  with 
$\psi\in C_c^\infty(I)$ 
 in  Definition \ref{def:strong}, one has
\begin{equ}
( u(t), 
\psi)_{L^2_x}=( u^{(0)},\psi)_{L^2_x} +\int_0^t(u^{[m]}(s), \Delta  \psi)_{L^2_x}\,ds+\sum_{k\in\N}\int_0^t( \sigma( u(s))e^k, \psi)_{L^2_x}\,dw^k_s
\end{equ}
for almost all $(t, \omega) 
\in [0,T] \times \Omega$.
Therefore, $u$  is a distributional solution of \eqref{eq:main}.  Moreover, the  Dirichlet boundary condition is encoded in the formulation in the following weak sense:
for all $s<t$, almost surely
\begin{equs}
\D^{-1}\Big( u(t)-u(s)-\sum_{k 
\in \N}  \int_s^t \sigma(u(r)) e^k \, dw^k_r \Big)  = \int_s^t u^{[m]}(r) \, dr .
 \end{equs}
 Notice that the left hand side of the above equality is an element of $H^1_x$ (in particular it vanishes on $\partial I $) and therefore so is the right hand side. 
\end{remark}

\begin{assumption}\label{as:sigma}
The function $\sigma:I\times \R\to\R$ is continuous and
there exist $\delta,K\geq 0$ such that for all $x\in I$, $r\in\R$, 
$$
|\sigma(x,r)|\leq K + \delta |r|^{(m+1)/2}
$$ 
\end{assumption}

Our first main result reads as follows. 
\begin{theorem}                   \label{thm:weak-existence}
For any $\gamma\in(-1,-1/2)$ there exists a $\delta_0=\delta_0(\gamma,m)$ such that the following holds.
Let $\sigma$ satisfy Assumption \ref{as:sigma}
with $\delta\leq\delta_0$
and let  $u^{(0)} \in L^{m+1}_\omega H^\gamma_x$. Then, there exists a weak solution  $\{(\bar\Omega,\bar \cF,\bar \P), \bar \F, (\bar w^k)_{k\in\N}, \bar u \}$ of equation \eqref{eq:main}. Moreover, $\bar{u}$ satisfies the following bounds:  
\begin{itemize}
\item[(i)] \label{item:estimates-space} (Energy estimates.) For all $p \in [0, m+1]$ there exists a constant $N= N(\gamma,p,m,K,T)$ such that 
\begin{equs}             
\bar \E\|\bar u\|_{L^\infty_tH^{\gamma}_x}^p
+\bar \E\|\bar u ^{[\frac{m+1}{2}]}\|_{L^2_tH^{1+\gamma}_x}^p+
\bar \E\|\bar u \|_{L^{m+1}_tW^{\gamma',m+1}_x}^{p(m+1)/2} \leq N( \E\|u^{(0)} \|_{H^\gamma}^p+1),
\\
 \label{eq:estimates-space}     
\end{equs}
where $\gamma'=\tfrac{2(1+\gamma)}{m+1}$.

\item[(ii)]\label{item:down-from-infinity} (Coming down from infinity.) There exists a constant $N=N(m,T)$ (in particular, independent of the initial condition) such that  for all  $t \in [0,T]$
\begin{equs}
\bar \E\|\bar u(t)\|_{H^{\gamma}_x}^2  \leq N t^{-2/(m-1)}. \label{eq:down-from-infinity}
\end{equs}
\item[(iii)] \label{eq:space-time-regularity} (Temporal regularity.)
For all $\eps>0$ there exists an $\eps'>0$ and a constant $N=N(\gamma,m,K,T)$ such that
\begin{equ}
\bar\E\|\bar u\|_{C^{\eps'}_tH^{\gamma-\eps}_x}^{\frac{m+1}{m}}\leq
N( \E\|u^{(0)} \|_{H^\gamma}^{m+1}+1).
\end{equ}
\end{itemize}
\end{theorem}

\begin{remark} \label{rem:growth}
In light of the smallness assumptions on $\delta$ above, it is worth noting that
if $\sigma$ is continuous and has polynomial growth with exponent $m'<(m+1)/2$,
then it satisfies Assumption \ref{as:sigma}  with arbitrarily small $\delta$ (and some appropriately chosen $K$). 
\end{remark}

\begin{remark}\label{rem:nondegenerate}
The reader may notice that the estimates \eqref{eq:estimates-space} are stronger than what the standard theory \cite{KR_SEE} yields for \emph{nondegenerate} quasilinear space-time white noise driven equations
\begin{equ}\label{eq:nondegenerate}
\d_t u=\d_x^2\big(A(u)\big)+\sigma(x,u)\xi.
\end{equ}
Indeed, if $A'$ takes values in $[\lambda,\lambda^{-1}]$ for some $\lambda>0$  and $\sigma$ is sufficiently smooth and small, then \cite{KR_SEE} provides well-posedness in the Gelfand triple $L^{2}_x\subset H^{-1}_x\equiv (H^{-1}_x)^*\subset (L^{2}_x)^*$, and therefore gives bounds in $C_tH^{-1}_x\cap L^2_{t,x}$.
We leave it as an exercise to the reader to check that our argument carries through for \eqref{eq:nondegenerate} and one can obtain the estimates \eqref{eq:estimates-space} with $m=1$, thus gaining (almost) $1/2$ regularity compared to \cite{KR_SEE}.
\end{remark}

In  Proposition  \ref{thm:strong-wellposedness} below we show that the monotone operator approach of \cite{P75, KR_SEE} can be applied to a small but nontrivial class of nonlinear diffusion coefficients to obtain (probabilistically) strong well-posedness.
Let us point out a key difference to \cite{BDPR}: Therein, assumptions on the drift operator to be coercive/monotone and the diffusion operator to be bounded/Lipschitz are considered separately. In contrast, by virtue of the elementary estimate from Lemma \ref{lem:krylov lemma} below, we use joint coercivity/monotonicity conditions for the drift and diffusion, in the spirit of the so-called stochastic parabolicity (see, e.g.\ \cite{P75, KR_SEE}).
Apart from the additive noise case, which directly follows from the monotone operator theory, the class of diffusion coefficients considered here contains some interesting cases - for example,
 $\sigma(r)=\lambda r^{[\frac{m+1}{2}]}$
 with sufficiently small $\lambda$ - but the conditions on $\sigma$ are certainly restrictive. The exponent $(m+1)/2$ guarantees that the noise shuts down sufficiently fast when the solution approaches zero, the region where the regularizing effect of the second order operator fades. 
\begin{assumption}\label{as:sigma2}
There exists  $\bar\delta$ such that for all $x\in I$, $r,\bar r\in\R$, 
$$
|\sigma(x,r)-\sigma(x,\bar r)|\leq \bar\delta|r^{[\frac{m+1}{2}]}-\bar r^{[\frac{m+1}{2}]}|.
$$ 
\end{assumption}

Under this additional assumption we have the next theorem. 
\begin{proposition}						\label{thm:strong-wellposedness}
Let $\sigma$ satisfy Assumptions \ref{as:sigma} and \ref{as:sigma2} with $\delta<6$ and $\bar\delta\leq 24\frac{m}{(m+1)^2}$,
and let $u^{(0)}\in L^2_\omega H^{-1}_x$.
Then there exists a unique strong solution to \eqref{eq:main}.
\end{proposition}

\begin{remark}
In continuation of Remark \ref{rem:growth}, Assumption \ref{as:sigma2}
requires higher exponents: if $\sigma(r)=r^{[m']}$ around $r=0$, then one needs $m'\geq (m+1)/2$.
For example, Assumption \ref{as:sigma2} excludes the linear multiplicative case $\sigma(u)=u$.
\end{remark}

\section{A priori estimates}\label{sec: a priori}

In this section we derive a priori estimates for approximations of \eqref{eq:main}.
Take some $\nu\in(0,1]$, $n\in\N$, and consider the equation
\begin{equ}\label{eq:approx}
d v =\d_x^2(\nu v+v^{[m]})\,dt+\sum_{k=1}^n \sigma(x,v) e^k d w^k_t
\end{equ}
on $Q$ with homogeneous Dirichlet boundary conditions and with initial condition $v_0=v^{(0)}$. 

\begin{assumption}\label{as:qualitative for approx}
The initial condition $v^{(0)}$ belongs to the space $L^{m+1}_{\omega,x}$ and $\sigma$ is Lipschitz continuous.

\begin{definition}
An $L^2$-solution equation \eqref{eq:approx} is an $\mathbb{F}$-adapted, continuous $L^2_x$-valued process $v$, such that almost surely $v, v^{[m]} \in L^2_t H^1_x$ and such that for all $\phi \in H^1_x$ we have with probability one
\begin{equ}
(v(t), 
\phi)_{L^2_x}=(v^{(0)},\phi)_{L^2_x} -\int_0^t\big(\partial _x (\nu v(s)+v^{[m]}(s)),  \partial _x \phi\big)_{L^2_x}\,ds+\sum_{k=1}^n\int_0^t\big( \sigma(v(s))e^k, \phi\big)_{L^2_x}\,dw^k_s.
\end{equ}
for all $t\in[0,T]$. 
\end{definition}

\end{assumption}
By \cite[Theorem 3.1, Remark 5.6]{DGT19}, under Assumption \ref{as:qualitative for approx}, equation \eqref{eq:approx} admits an $L^2$-solution
$v$. Moreover, on the basis of It\^o's formula for the functions 
$$
u \mapsto \|u\|^2_{L^2_x}, \qquad u  \mapsto \|u\|^{m+1}_{L^{m+1}_x}
$$ 
one can easily derive that  for $p\in[0, m+1]$
\begin{equ}\label{eq:approx bound}
\E\|v\|_{C_tL^2_x}^p+\E\|v\|_{L^2_t H^1_x}^p+\E \|v^{[\frac{m+1}{2}]}\|_{L^2_tH^1_x}^p+ \E \|v^{[m]}\|_{L^2_t H^1_x}^p<\infty.
\end{equ}

The main result of this section is the following.
\begin{lemma}\label{lem:a priori}
For any $\gamma\in[-1,-1/2)$ there exists a $\delta_0=\delta_0(\gamma,m)$ such that the following holds.
Let $\sigma$ and $v^{(0)}$ satisfy Assumption \ref{as:qualitative for approx},
$\sigma$ satisfy Assumption \ref{as:sigma} (a) with $\delta\leq \delta_0$
and $v^{(0)}\in L^{m+1}_\omega H^\gamma_x$.
Then, for all $p\in[0,m+1]$ the solution $v$ of \eqref{eq:approx} satisfies the bound
\begin{equ}
\label{eq:a priori 1}
\E\|v\|_{L^\infty_tH^{\gamma}_x}^p
+\E\|v^{[\frac{m+1}{2}]}\|_{L^2_tH^{1+\gamma}_x}^p
\leq N(\E\|v^{(0)}\|_{H^\gamma_x}^p+1),  
\end{equ}
with some $N=N(\gamma,p,m,K,T)$.
Moreover, with the notation $\gamma'=\tfrac{2(1+\gamma)}{m+1}$, one also has the bound
\begin{equ}\label{eq:a priori 2}
\E\|v\|_{L^{m+1}_tW^{\gamma',m+1}_x}^{p(m+1)/2}
\leq N(\E\|v^{(0)}\|_{H^\gamma_x}^p+1).
\end{equ}
\end{lemma}

First we collect some auxiliary statements.
The following lemma is part of \cite[Lem~8.4]{K_Lp} with $\R$ in place of $I$ and $(1-\Delta)$ in place of $\D$.
In our setting the proof is particularly short, so we include it for the sake of completeness.
\begin{lemma}\label{lem:krylov lemma}
For all $\tilde \gamma<-1/2$ there exists a constant $N=N(\tilde \gamma)$ such that for all $u\in L^2_x$ one has
\begin{equ}\label{eq:krylov bound}
\sum_{k\in\N}\|ue^k\|_{H^{\tilde \gamma}_x}^2\leq N \|u\|_{L^2_x}^2.
\end{equ}
Moreover, one has $N(-1)=1/3$.
\end{lemma}
\begin{proof}
By the definition of the norm in $H^{\tilde \gamma}_x$ and Parseval's identity we get
\begin{equs}
\sum_{k\in\N}\|ue^k\|_{H^{\tilde \gamma}_x}^2= \sum_{k\in\N} \sum_{l\in\N}\lambda_l^{\tilde \gamma}|(u e^k, e^l)_{L^2_x}|^2= \sum_{l\in\N}\lambda_l^{\tilde \gamma}\|u  e^l\|^2_{L^2_x} \leq N \|u\|^2_{L^2_x},
\end{equs}
where we have used that for all $l\in\N$, $\|e^l\|_{L^\infty_x}=\sqrt{2}$ and $\lambda_l= (\pi l)^2$.
\end{proof}


Next is a  bound to deduce the regularity of a function from the regularity of its monotone power.

\begin{lemma}\label{lem:power regularity}
Let $u$ be such that $u^{[\tilde m]}\in H^{\tilde\gamma}_x$ with $\tilde\gamma\in[0,1/2)$ and $\tilde m\in(1,\infty)$.
Then $u\in W^{\tilde \gamma/\tilde m,2\tilde m}_x$ and there exists a constant $N=N(\tilde m)$ such that the following bound holds
\begin{equ}\label{eq:power regularity}
\|u\|_{ W^{\tilde \gamma/\tilde m,2\tilde m}_x}^{2\tilde m}\leq N\|u^{[\tilde m]}\|_{H^{\tilde \gamma}_x}^2.
\end{equ}
\end{lemma}
\begin{proof}
The $\tilde \gamma=0$ case is trivial. For $\tilde\gamma\in(0,1/2)$ we use the
equivalent norm \eqref{eq:norm} and the
 elementary inequality $|a-b|^{2\tilde m}\leq N(\tilde m)|a^{[\tilde m]}-b^{[\tilde m]}|^2$  to write
\begin{equs}
\|u\|_{W^{\tilde\gamma/\tilde m,2\tilde m}_x}^{2\tilde m}
&\leq N\|u\|_{L^{2\tilde m}_x}^{2\tilde m}+  N\int_{I\times I}\frac{|u(x)-u(y)|^{2\tilde m}}{|x-y|^{1+2\tilde m(\tilde\gamma/\tilde m)}}\,dy\,dx
\\
&\leq N\|u^{[\tilde m]}\|_{L^{2}_x}^{2}+ N\int_{I\times I}\frac{|u(x)^{[\tilde m]}-u(y)^{[\tilde m]}|^2}{|x-y|^{1+2\tilde\gamma}}\,dy\,dx
\\&\leq N\|u^{[\tilde m]}\|_{ W^{\tilde \gamma,2}_x}^2
\leq N\|u^{[m]}\|_{H^{\tilde \gamma}_x}^2.
\end{equs}
\end{proof}

The next tool is the Stroock-Varopoulos lemma.
It appears and is proved in various forms in the literature, so for the convenience of the reader we give the proof in the appendix, following the standard proof strategy (see e.g. \cite[Sec~5]{Vazquez_fractional}) based on the Cafarelli-Silvestre extension \cite{C-S}.
An alternative, more elementary strategy can be found in e.g. \cite[App~B2]{Figalli}.
\begin{lemma}[Stroock-Varopoulos]\label{lem:stroock-varopoulos}
Let $\beta \in(0,1/2)$ and let $f,g\in C^1(\R)$ satisfy $f'=(g')^2$. Then, for any $u \in H^1$ 
such that $f(u),g(u) \in H^1$, we have
\begin{equ}\label{eq:stroock-varopoulos}
\int_I  f(u) \D^\beta u \,dx
\geq
\int_I |\D^{\beta/2} g(u)|^2(x)\,dx.
\end{equ}
\end{lemma}
The particular form that we use below is
\begin{equ}\label{eq:s-v_power}
\int_I  u^{[m]} \D^\beta u \,dx
\geq \frac{4  m}{( m+1)^2}
\int_I |\D^{\beta/2} u^{[\frac{ m+1}{2}]}|^2(x)\,dx,
\end{equ}
the more general form \eqref{eq:stroock-varopoulos} can be useful for general nonlinear leading operators, for example in the context of Remark \ref{rem:nondegenerate}. We can now prove the a priori bounds \eqref{eq:a priori 1}-\eqref{eq:a priori 2}.
\begin{proof}[Proof of Lemma \ref{lem:a priori}]

We test the equation with $e^l$ and apply It\^o's formula for the square to obtain 
\begin{equs}
|(v(s), e^l)_{L^2_x}|^2& = |(v^{(0)}, e^l)_{L^2_x}|^2-\int_0^s 2 \lambda_l \big( \nu v(r)+ v^{[m]}(r) , e^l\big)_{L^2_x}(v(r), e^l)_{L^2_x} \, dr  
\\
&+ \sum_{k=1}^n\int_0^s  |\big(\sigma(v(r))e^k, e^l\big)_{L^2_x}|^2 \,  dr
\\
& +  \sum_{k=1}^n 2 \int_0^s  \big(\sigma(v(r))e^k, e^l\big)_{L^2_x} (v(r), e^l)_{L^2_x}\, dw^k_r.
\end{equs}
We multiply the above equality with $\lambda_l^\gamma$ and we sum over $l$ to obtain  
\begin{equs}[eq:Ito]
\|v(s)\|_{H^{\gamma}_x}^2&=
\|v^{(0)}\|_{H^{\gamma}_x}^2
-2\nu\int_0^s\big(v(r),\D^{1+\gamma}v(r)\big)_{L^2_x}\,dr
\\
&-2\int_0^s \big(v^{[m]}(s),\D^{1+\gamma}v(r)\big)_{L^2_x}\,dr
+\int_0^s\sum_{k=1}^n \|\sigma(v(r)) e^k\|_{H^\gamma_x}^2\,dr
\\
&+\sum_{k=1}^n 2\int_0^s\big(\D^{\gamma/2} v(r),\D^{\gamma/2}
 (\sigma(v(r))e^k)\big)_{L^2_x}\,dw^k_r.
\end{equs}
Denote the last term by $M_s$. 
The first integral on the right-hand side is nonnegative so we simply bound it by $0$.
For the second one
we apply \eqref{eq:s-v_power} (with $\beta=\gamma$)
and for the third
we use Lemma \ref{lem:krylov lemma} (with $\tilde\gamma=\gamma$).
We therefore get
\begin{equs}
\|v(s)\|_{H^{\gamma}_x}^2
&\leq
\|v^{(0)}\|_{H^{\gamma}_x}^2
-\tfrac{8m}{(m+1)^2}\|v^{[\frac{m+1}{2}]}\|_{L^2([0,s];H^{1+\gamma}_x)}^2+
\bar N(\gamma)\|\sigma(v)\|_{L^2([0,s];L^2_x)}^2+M_s
\\ & \leq
\|v^{(0)}\|_{H^{\gamma}_x}^2
-\tfrac{8m}{(m+1)^2}\|v^{[\frac{m+1}{2}]}\|_{L^2([0,s];H^{1+\gamma}_x)}^2
\\
&  \qquad  \qquad +
\bar N(\gamma)\delta\|v\|_{L^{m+1}([0,s];L^{m+1}_x)}^{m+1}+N+M_s.
\end{equs}
Notice that $\tfrac{8m}{(m+1)^2}>\tfrac{2}{m}$.
Assuming $\delta$ is small enough so that $\bar N(\gamma)\delta\leq \tfrac{1}{m}$, we obtain
\begin{equ}[eq:energy1]
\|v(s)\|_{H^{\gamma}_x}^2 +\tfrac{1}{m}\|v^{[\frac{m+1}{2}]}\|_{L^2([0,s];H^{1+\gamma}_x)}^2
\leq  \|v^{(0)}\|_{H^{\gamma}_x}^2 +N+M_s.
\end{equ}
The quadratic variation process $\langle M\rangle$ of the local martingale $M$ is given by
\begin{equs}[eq:quadratic var]
\langle M\rangle_s
&=4\int_0^s \sum_{k=1}^n\big(\D^{\gamma/2} v(r),\D^{\gamma/2} (\sigma(v(r))e^k)\big)_{L^2_x}^2\,dr
\\
&\leq 4 \int_0^s \|v(r)\|_{H^\gamma_x}^2\sum_{k\in\N}\|\sigma(v(r))e^k\|_{H^\gamma_x}^2\,dr
\\
&\leq
4 \bar  N(\gamma)\int_0^s \|v(r)\|_{H^\gamma_x}^2\|\sigma(v(r))\|_{L^2_x}^2\,dr
\\
&\leq 4 \bar N(\gamma)\delta 
\|v\|_{L^\infty([0,s];H^\gamma_x)}^2\|v\|_{L^{m+1}([0,s];L^{m+1}_x)}^{m+1}
+N \|v\|_{L^2([0,s];H^\gamma_x)}^2
\end{equs}
where we used Lemma \ref{lem:krylov lemma} and the growth of $\sigma$ as before.
In particular, by \eqref{eq:approx bound}, $M$ is a martingale.
Denote $X_s=\|v\|_{L^\infty([0,s],H^\gamma_x)}$ and $Y_s=\|v^{[\frac{m+1}{2}]}\|_{L^2([0,s];H^{1+\gamma}_x)}$.
By \eqref{eq:energy1}, \eqref{eq:quadratic var} and the Burkholder-Gundy-Davis and Jensen inequalities, we have for any stopping time $\tau \leq T$, 
\begin{equs}
\E X_{s\wedge \tau} ^{m+1} +\tfrac{1}{m^{(m+1)/2}}\E Y_{s\wedge \tau}^{m+1}
&\leq N \E\|v^{(0)}\|_{H^\gamma_x}^{m+1} + N + N \int_0^s\E X_{r\wedge \tau}^{m+1}\,dr
\\
&\qquad
+\tilde N(\gamma,m)\delta^{(m+1)/4} \E (X_{s\wedge \tau}^{(m+1)/2} \tfrac{1}{m^{(m+1)/4}}Y_{s\wedge \tau}^{(m+1)/2}).
\end{equs}
By \eqref{eq:approx bound}, we have that $\E X_ {s\wedge \tau} ^{m+1} +\E Y_ {s\wedge \tau} ^{m+1}<\infty$.
Therefore, assuming $\delta$ is small enough so that $\tilde N(p,m)\delta^{(m+1)/4}\leq 1$, we can apply Young's inequality for the last term and absorb it in the left-hand side.
We get
\begin{equ}
\E X_{s\wedge \tau}^{m+1} +\E Y_ {s\wedge \tau}^{m+1}\leq N \E\|v^{(0)}\|_{H^\gamma_x}^{m+1} + N + N\int_0^s\E X_{r\wedge \tau}^{m+1}\,dr.
\end{equ}
Applying Gronwall's inequality for the function $s\mapsto\E X_{s\wedge \tau}^{m+1} +\E Y_{s\wedge \tau}^{m+1}$, yields
\begin{equ}
\E X_\tau ^{m+1}+ \E Y^{m+1}_\tau \leq N\E\|v^{(0)}\|_{H^\gamma_x}^{m+1} + N.
\end{equ}
By choosing $\tau=T$ we have  \eqref{eq:a priori 1} for $p=m+1$. The result for $p<m+1$ follows by Lenglart's inequality (see, e.g., \cite[Proposition IV.4.7 and Exercise IV.4.31/1]{Karatzas} ).  Finally, the bound \eqref{eq:a priori 2} then follows by applying Lemma \ref{lem:power regularity} with $\tilde\gamma=1+\gamma$, $\tilde m=(m+1)/2$.
\end{proof}

\begin{corollary}\label{cor:extra power}
Take $\gamma\in(-1,-1/2)$. Let $\sigma$ and $v^{(0)}$ satisfy Assumption \ref{as:qualitative for approx},
$\sigma$ satisfy Assumption \ref{as:sigma}  with $\delta\leq \delta_0(\gamma,m)$.
Then, there exists a $c=c(\gamma,m)>1$ such that
\begin{equ}\label{eq:extra power}
\|v\|_{L_{\omega,t,x}^{c(m+1)}}\leq N(\|v^{(0)}\|_{L^{m+1}_\omega H^\gamma_x}+1)
\end{equ}
with some $N=N(\gamma,p,m,K,T)$.
\end{corollary}
\begin{proof}
By the standard interpolation properties of $L^p$ spaces and \eqref{eq:interpolation}
\begin{equs}
\| v\|_{L^{\frac{(m+1)^2}{2}- \eps_1(\theta)}_{\omega} L^{m+1+\eps_2(\theta)}_t W_x^{\gamma'-\eps_3(\theta), m+1-\eps_4(\theta)}}
&\leq  N \|v\|_{L^{m+1}_{\omega} L^\infty_tW^{\gamma,2}_x}^{1-\theta} 
\|v\|_{L^{\frac{(m+1)^2}{2}}_{\omega} L^{m+1}_tW^{\gamma',m+1}_x}^{\theta}
\\
&=
N \|v\|_{L^{m+1}_{\omega} L^\infty_t H^{\gamma}_x}^{1-\theta}
\|v\|_{L^{\frac{(m+1)^2}{2}}_{\omega} L^{m+1}_tW^{\gamma',m+1}_x}^{\theta}
,
\end{equs}
where in the last step we used Proposition \ref{prop:Sobolev} \eqref{p:spectral} and \eqref{p:dual}.
Here $\theta\in (0,1)$ and $\eps_i(\theta)>0$ such that $\eps_i(\theta)\to 0$ as $\theta\to 1$.
Since $\gamma>-1$ implies $\gamma'>0$, we can choose $\theta$ sufficiently close to $1$ such that
\begin{equ}
\frac{(m+1)^2}{2}- \eps_1(\theta)>m+1,\qquad
\gamma'-\eps_3(\theta)-\frac{1}{m+1-\eps_4(\theta)}>-\frac{1}{m+1}.
\end{equ}
By Sobolev's embedding we then see that for some $c>1$,
\begin{equ}
\|v\|_{L_{\omega,t,x}^{c(m+1)}}\leq
N\| v\|_{L^{\frac{(m+1)^2}{2}- \eps_1(\theta)}_{\omega} L^{m+1+\eps_2(\theta)}_t W_x^{\gamma'-\eps_3(\theta), m+1-\eps_4(\theta)}}.
\end{equ}
On the other hand, \eqref{eq:a priori 1}-\eqref{eq:a priori 2} yields
\begin{equ}
\|v\|_{L^{m+1}_{\omega} L^\infty_t H^{\gamma}_x}^{1-\theta}
\|v\|_{L^{\frac{(m+1)^2}{2}}_{\omega} L^{m+1}_tW^{\gamma',m+1}_x}^{\theta}
\leq N(\|v^{(0)}\|_{L^{m+1}_\omega H^\gamma_x}+1)^{(1-\theta)+\theta\tfrac{2}{m+1}},
\end{equ}
and putting the above bounds together we readily get \eqref{eq:extra power}.
\end{proof}


\subsection{Time regularity}

The following is a simple variant of \cite[Lem~2.1]{Flandoli-Gatarek}. While in \cite{Flandoli-Gatarek} only the $q'=q$ case is stated, the form below is easily obtained via Lenglart's inequality as before.
\begin{lemma}\label{lem:time regularity}
Let $H$ be a separable Hilbert space, $q\geq 2$, and $f=(f^k)_{k=1}^\infty$ be a progressively measurable $\ell^2(H)$-valued process such that $f\in L^q(\Omega\times[0,T],\ell^2(H))$.
Then, for all $\alpha<1/2$ and $q'\in[0,q]$ there exists $N=N(\alpha,q,q')$ such that
\begin{equ}
\E\|t\mapsto\sum_{k\in\N}\int_0^tf^k(s)\,dw^k_s\|_{W^{\alpha,q}([0,T],H)}^{q'}
\leq N\E\|f\|_{L^q([0,T],\ell^2(H))}^{q'}.
\end{equ}
\end{lemma}

\begin{corollary}
Take $\gamma\in(-1,1/2)$.
Let $\sigma$ and $v^{(0)}$ satisfy Assumption \ref{as:qualitative for approx},
$\sigma$ satisfy Assumption \ref{as:sigma} 
with $\delta\leq \delta_0(\gamma,m)$.
Let, furthermore, $\alpha< 1/2$,  $\beta>5/2$, and define the space
\begin{equ}\label{eq:def X}
\cX= W^{1,\frac{m+1}{m}}\big([0,T]; H^{-\beta}(I))+W^{\alpha, 2c}([0,T];H^{-1}(I)\big),
\end{equ}
where $c$ is as in Corollary \ref{cor:extra power}.
 Then, the solution $v$ of \eqref{eq:approx} satisfies the bound
\begin{equs}\label{eq:time-reg} 
\E\|v\|_\cX^{\frac{m+1}{m}}\leq N(\E\|v^{(0)}\|_{H^{\gamma}_x}^{m+1}+1)
\end{equs}
with some $N=N(\alpha,\beta, m,K,T)$.
\end{corollary}
\begin{proof}
We apply Lemma \ref{lem:time regularity} with $f^k(s)=\sigma(v(s))e^k\one_{k\leq n}$, $H=H^{-1}_x$, $q=2c$, and $q'=2$,
to get that
\begin{equs}[eq:reg-stoch]
\E & \|s\mapsto\sum_{k=1}^n\int_0^s \sigma(v(r)) e^k\,d w^k_r\|_{W^{\alpha, 2c}_t H^{-1}_x}^2
\\
&\leq
N \E\left(\int_0^T\Big(\sum_{k=1}^n\|\sigma(v(s))e^k\|_{H^{-1}_x}^{2}\Big)^{c}\,ds\right)^{1/c}
\\& \leq
N(\E \|v^{[\frac{m+1}{2}]}\|_{L^{2c}_{t}L^2_x}^{2}+ 1)
\\
&
=N(\E\|v\|_{L_t^{c(m+1)}L_x^{m+1}}^{m+1} +1),        
\end{equs}
where we have used Lemma \ref{lem:krylov lemma} and the growth of $\sigma$ in the last inequality.
On the other hand, one easily sees that
\begin{equs}[eq:reg-det]
 \E & \|s\mapsto\int_0^s\Delta\big(\nu v(r)+v^{[m]}(r)\big)\,dr\|_{W^{1,\frac{m+1}{m}}_tH^{-\beta}_x}^{\frac{m+1}{m}}
\\
&\leq N \E 
\int_0^T \| \Delta\big(\nu v(r)+v^{[m]}(r)\big) \|_{H^{-\beta}_x}^{\frac{m+1}{m}} \, dt
\\
&\leq N\E 
\int_0^T \Big(\sum_{k\in\N} \lambda_k^{2-\beta}\big(\nu v(r)+ v^{[m]}(r),e^k\big)^2\Big)^{\frac{m+1}{2m}}\,dt
\\
&\leq N \big(\E\|v^{[\frac{m+1}{2}]}\|_{L^{2}_{t,x}}^{2}+1\big)\left( \sum_{k\in\N}k^{4-2\beta}\right) ^{\frac{m+1}{2m}}.
\end{equs}
Since $\beta>5/2$, the last sum is finite. By \eqref{eq:extra power}, the right-hand-side of both \eqref{eq:reg-stoch} and \eqref{eq:reg-det} are bounded as in \eqref{eq:time-reg}, hence the proof is finished.
\end{proof}

\section{Limiting procedure}\label{sec:limiting}

\begin{proof}[Proof of Theorem \ref{thm:weak-existence}]
Let $\sigma_n : I\times \R \to \R$ be bounded smooth functions with bounded derivatives such that 
 $\sigma_n \to \sigma$ uniformly on compacts  as $n \to \infty$ and  for all $x\in I$, $r \in \R$
\begin{equs}                       \label{eq:bound-sigma-n}
\sup_n |\sigma_n(x,r)| \leq K+ \delta |r|^{(m+1)/2},
\end{equs}
  and let $u^{(0)}_n$ be $\mathcal{F}_0$-measurable random variables with $\E \|u^{(0)}_n\|_{L^{m+1}_x}^{m+1}< \infty$  and 
 \begin{equs}                      \label{eq:convergence-initial-conditions}
 \lim_{n \to \infty} \E \|u^{(0)}_n- u^{(0)}\|^{m+1}_{H^\gamma_x} = 0.
 \end{equs}
 Let $u_n$ be an $L^2$-solution of 
\begin{equs} [eq:approximation-of-main]
d u_n & = \d^2_x ( n^{-1} u_n+u_n^{[m]}) \, dt + \sum_{k=1}^n \sigma_n(x,u_n)e^k \, dw^k
\\
u_n(0)&=u_n^{(0)}.
\end{equs}
Take $\alpha\in(1/(2c),1/2)$, where $c$ is as in Corollary \ref{cor:extra power},   $\beta\in(5/2,3)$, and set $\cX$ as in \eqref{eq:def X}.
Let us set 
$$
\mathcal{Y}= L^{{m+1}}_tW^{\gamma',m+1}_x  \cap W^{\alpha, \frac{m+1}{m}}_t H^{-\beta}_x,
$$
where $\gamma'= 2(1+\gamma)/(m+1)$, as in Lemma \ref{lem:a priori}.
By  \cite[Thms~2.1-2.2]{Flandoli-Gatarek} we have the compact embeddings
\begin{equ}\label{eq:embeddings}
\mathcal{Y} \Subset L^{{m+1}}_{t,x},\qquad \mathcal{X} \Subset C_t H^{-3}_x. 
\end{equ}
Therefore,
$$
\mathcal{Y}\cap \mathcal{X}  \Subset 
L^{{m+1}}_{t,x}\cap C_t H^{-3}_x
=: \mathcal{Z}.
$$
Notice that $\cX\subset W^{\alpha, \frac{m+1}{m}}_t H^{-\beta}_x$. Therefore,
by  \eqref{eq:a priori 2} and \eqref{eq:time-reg} we have the estimate
\begin{equs}\label{eq:a-priori-main}
\E\|u_n\|_{\mathcal{X} \cap \mathcal{Y}}^{\frac{m+1}{m}} \leq N( \E \|u^{(0)}_n\|^{m+1}_{H^\gamma_x}+1)\leq N( \E \|u^{(0)}\|^{m+1}_{H^\gamma_x}+1),
\end{equs}
which in turn implies that the laws of $(u_n)_{n\in\N} $ on $\mathcal{Z}$ are tight. 
Let us set 
$$
w(t)= \sum_{k\in\N}\frac{1}{\sqrt{2^k}}w^k(t)\mathfrak{e}_k,
$$
where $(\mathfrak{e}_k)_{k\in\N}$ is the standard orthonormal basis of $\ell^2$.
By Prokhorov's theorem, there exists a (non-relabelled) subsequence $(u_n)_n$ such that the laws of $
(u_n, w)$ on $\mathcal{Z} \times C([0,T]; \ell^2)$ are weakly convergent. By Skorohod's representation theorem, there exist $\mathcal{Z} \times C([0,T]; \ell^2)$-valued random variables $(\bar u,\bar w) $, $(\bar u_n, \bar w_n )$, for $n \in \N$, on a probability space $(\bar \Omega, \bar {\mathcal{F}}, \bar{\P})$,
such that in $\mathcal{Z} \times C([0,T]; \ell^2)$, $\bar{\P}$-almost surely
\begin{equation}          \label{eq:convergence-in-Z}
(\bar u_n,\bar w_n )\to (\bar u , \bar w),
\end{equation}
as $n  \to \infty$, and for each $n \in \N$, as random variables in $\mathcal{Z}\times C([0,T]; \ell^2)$
\begin{equation}            \label{eq:distribution}
(\bar u_n, \bar  w_n )\overset{d}{=}(u_n,w).
\end{equation}
 Moreover, upon passing to a subsequence, we may assume that 
 \begin{equation}                     \label{eq:almost-everywhere}
 \bar u_n  \to\bar u \  \ \text{for almost all \ $(\bar{\omega}, t,x)$}.
  \end{equation}
Let $(\bar{\mathcal{F}}_t)_{t \in [0,T]}$ be the augmented filtration of
$\mathcal{G}_t:= \sigma( \bar u(s), \bar w (s); s \leq t)$, and let $\bar w ^k(t):= \sqrt{2^k}(\bar w(t), \mathfrak{e}_k)_{\ell^2}$. It is easy to see that $\bar w^k$, $k \in \N$, are independent, standard, real-valued $\bar{\mathcal{F}}_t$-Wiener processes (see for example the argument in \cite[Proof of Prop. 5.5]{DGT19}).

We now show that $\bar u$ is a weak solution. 
Notice that by virtue of the a priori estimates \eqref{eq:a priori 1}-\eqref{eq:a priori 2} and \eqref{eq:distribution} we have  
\begin{equs}                          \label{eq:bounds-bar-un}
\bar \E\|\bar u_n\|_{L^\infty_tH^{\gamma}_x}^2
+\bar \E\|\bar u_n^{[\frac{m+1}{2}]}\|_{L^2_tH^{1+\gamma}_x}^2+
\bar \E\|\bar u_n\|_{L^{m+1}_tW^{\gamma',m+1}_x}^{m+1} \leq N( \E\|u^{(0)}\|_{H^\gamma_x}^2+1),
\end{equs}
which, by the lower semicontinuity of the norms, gives 
\begin{equs}                     \label{eq:bounds-bar-u}
\bar \E\|\bar u\|_{L^\infty_tH^{\gamma}_x}^2
+\bar \E\|\bar u ^{[\frac{m+1}{2}]}\|_{L^2_tH^{1+\gamma}_x}^2+
\bar \E\|\bar u \|_{L^{m+1}_tW^{\gamma',m+1}_x}^{m+1} \leq N( \E\|u^{(0)} \|_{H^\gamma_x}^2+1).
\end{equs}
Let us set 
\begin{equs}
\nonumber
M(\bar u, t) &:= \bar{u}(t)- \bar{u}(0)-\int_0^t  \Delta  \big( \bar u^{[m]}(s) \big)  \, ds,
\end{equs}
and for $v\in \{u_n, \bar u_n\}$,
\begin{equs}
M_n(v, t) &:=  v(t)- v(0)-\int_0^t \Delta \big( n^{-1} \bar v(s) + \bar v^{[m]}(s) \big)  \, ds.
\end{equs}
Fix an arbitrary $l \in \mathbb{N}$. We will show that for any $\phi \in H^{-3}_x$, the processes 
\begin{equs}
 M^1(\bar u,  t)&:= (M(\bar u,  t), \phi)_{H^{-3}_x},
\\ \label{eq:def-M2}
 M^2( \bar u , t)&:= ( M( \bar u , t), \phi)^2_{H^{-3}_x}-\int_0^t \sum_{k\in\N} | \big(\sigma(\bar u(s)) e^k , \phi\big)_{H^{-3}_x}|^2 \, ds,
\\
\bar  M^{3}( \bar u , t)&:=\bar w^l(t)(  M( \bar u , t), \phi)_{H^{-3}_x}- \int_0^t \big(\sigma(\bar u(s))e^l, \phi\big)_{H^{-3}_x} \, ds
\end{equs}
are continuous $\bar {\mathcal{F}}_t$-martingales. We first show that they are continuous $\mathcal{G}_t$-martingales. Assume for now that $\phi \in C^\infty_c(I)$. For, $i=1,2,3$ and $v \in \{ u_n , \bar u_n\}$,  let us also define the processes $M^i_n(v, t)$ similarly to $ M^i( \bar u, t)$, but with $ \bar u, M( \bar u, t), \sigma(\bar u)$ replaced by $v, M_n(v,t), \sigma_n (v) $, and the corresponding summation in \eqref{eq:def-M2} going only up to $n$.    Let us fix $s < t$ and let $V$ be a bounded, continuous function on $C([0,s];H^{-3}_x) \times C([0, s] ; \ell^2)$. We have that 
\begin{equs}
(M_n(u_n, t), \phi)_{H^{-3}_x}&= \sum_{k =1}^n \int_0^t \big(\sigma_n(u_n(s))e^k, \phi\big)_{H^{-3}_x} \, d w^k_s.
\end{equs}
It follows that $M^i_n(u_n, t)$ are continuous $\mathcal{F}_t$-martingales. Hence, for $i=1,2,3$, 
$$
\E V(u_n|_{[0,s]},w|_{[0,s]})(M_n^i(u_n , t)-M_n^i(u_n, s))=0,
$$
which combined with \eqref{eq:distribution} gives 
\begin{equation}                \label{eq:martingale-Ml}
\bar \E V( \bar u_n|_{[0,s]},\bar w_n |_{[0,s]}) (M_n^i(\bar u_n, t)-  M_n^i(\bar u_n, s)) =0.
\end{equation}
Next, notice that $\bar \P$-almost surely
\begin{equs}
  \int_0^T\big| & \big( \Delta  (n^{-1}\bar u_n(t)+ \bar u_n^{[m]}(t)-\bar u^{[m]}(t)) , \phi \big)_{H^{-3}_x} \big| \,   dt 
\\
=&  \int_0^T \Big|\sum_{k\in\N}  \lambda_k^{1-3}\big(n^{-1}\bar u_n(t)+ \bar u_n^{[m]}(t)-\bar u^{[m]}(t), e^k\big)_{L^2_x} (\phi, e^k)_{L^2_x} \, \Big| \, dt  
\\
\leq  & N \|  \phi \|_{L^\infty_x}   (n^{-1}\|\bar u_n \|_{L^1_{t,x}}+  \| \bar u_n^{[m]}-\bar u^{[m]}\|_{L^1_{t,x}})  \to 0,         \label{eq:conPhi}
\end{equs}
where the convergence follows from  \eqref{eq:almost-everywhere} and the bounds \eqref{eq:bounds-bar-un}. 
Hence, by \eqref{eq:conPhi}  and \eqref{eq:convergence-in-Z} we see that for each $t \in [0,T]$, $\bar \P$-almost surely
\begin{equation}         \label{eq:MltoM-in-probability}
( M_n(\bar u_n ,t), \phi)_{H^{-3}_x}\to (M(\bar u , t), \phi)_{H^{-3}_x}.
\end{equation} 
In addition, it is easy to see that 
\begin{equs}
 \bar \E \int_0^T \Big| & \sum_{k=1}^n \big(\sigma_n( \bar u_n(t)) e^k , \phi\big)^2_{H^{-3}_x}- \sum_{k\in\N} \big(\sigma( \bar u(t)) e^k , \phi\big)^2_{H^{-3}_x} \Big| \, dt
 \\
 \leq & N \|\phi\|_{H^{-3}_x}^2 \bar  \E \int_0^T \sum_{k=n+1}^\infty \|\sigma( \bar u (t) ) e^k\|^2_{H^{-3}_x}\, dt
 \\
   &\quad+  N \|\phi\|_{H^{-3}_x}^2 \bar \E \int_0^T \|\sigma( \bar u(t)  )-\sigma_n(\bar u_n(t))\|_{L^2_x}\|\sigma( \bar u (t) )+\sigma_n(\bar u_n(t))\|_{L^2_x}\, dt,                
   \\
    \label{eq:convergence-squares}
\end{equs}
where we have used also Lemma \ref{lem:krylov lemma}. The first term of the right hand side converges to zero as $n \to \infty$  by virtue of Lemma \ref{lem:krylov lemma}, Assumption \ref{as:sigma} and \eqref{eq:bounds-bar-u}.  
For the second term we have 
\begin{equs}
 \Big| \bar \E \int_0^T &\|\sigma( \bar u (t) )-\sigma_n(\bar u_n(t))\|_{L^2_x}\|\sigma( \bar u (t) )+\sigma_n(\bar u_n(t))\|_{L^2_x}\, dt \Big|^2
\\
 \leq &  \bar \E \|\sigma( \bar u  )-\sigma_n(\bar u_n) \|^2_{L^2_{t,x}}  (\bar \E \| \bar u\|^{m+1}_{L^{m+1}_{t,x}}+\bar \E \|\bar u_n\|^{m+1}_{L^{m+1}_{t,x}}+1)
 \\
 \leq & N \bar  \E \|\sigma( \bar u  )-\sigma_n(\bar u_n) \|^2_{L^2_{t,x}} ( \E\|u^{(0)} \|_{H^\gamma_x}^2+1),
\end{equs}
where we have used Assumption \ref{as:sigma}, \eqref{eq:bound-sigma-n} and the bounds \eqref{eq:bounds-bar-un}-\eqref{eq:bounds-bar-u}.
By \eqref{eq:almost-everywhere}, the uniform convergence on compacts of $\sigma_n$ to $\sigma$  and the continuity of $\sigma$ we have that $|\sigma_n(\bar u_n)- \sigma ( \bar u)|^2 \to 0  $ for almost every $(\bar \omega, t, x)$.  Moreover, by Assumption \ref{as:sigma}   and \eqref{eq:bound-sigma-n}, we have 
\begin{equs}
|\sigma_n( \bar u_n ) - \sigma( \bar u ) |^2 \leq N (1+ |\bar u_n|^{m+1}+ |\bar u|^{m+1}).
\end{equs}
Hence, to conclude that the right hand side of \eqref{eq:convergence-squares} converges to zero, it suffices to check that $|\bar u_n|^{m+1}$ are uniformly integrable in $(\bar\omega, t, x)$.
This follows immediately from \eqref{eq:extra power} and \eqref{eq:distribution}.
Using \eqref{eq:MltoM-in-probability} we can conclude that $M^2_n(\bar u_n, t) \to M^2(\bar u, t)$ in probability. 
Similarly one shows  that  $M^3_n(\bar u_n, t) \to M^3(\bar u, t)$. Therefore, for each $t \in [0,T]$ we have that $M^i_n(\bar u_n, t) \to M^i(\bar u, t)$ in probability.
 Moreover, with $c>1$ from Corollary \ref{cor:extra power},  we have 
\begin{equs}
\nonumber
\sup_{n\in\N} \bar{\E}| (M(\bar u_n ,t), \phi)_{H^{-3}_x}|^{2c}&= \sup_{n\in\N} \E \left|  \sum_{k=1}^n  \int_0^t \big(\sigma(u_n(s))e^k , \phi\big)_{H^{-3}_x} \, dw^k_s\right|^{2c}
\\
& \leq N \|\phi\|_{H^{-3}_x}^{2 c} \sup_{n\in\N} ( 1+ \| u_n \|^{c(m+1)}_{L^{c (m+1)}_{ \omega, t, x}})< \infty.
\end{equs}
and 
\begin{equs}
\sup_{n\in\N} \bar \E \left| \int_0^t \sum_{k=1}^n \big( \sigma(\bar u_n(s) ) e^k , \phi\big)_{H^{-3}_x}^2 \, ds\right|^c  \leq N \| \phi\|^{2c}_{H^{-3}_x}\sup_{n\in\N} ( 1+ \|\bar u_n \|^{c(m+1)}_{L^{c (m+1)}_{\bar \omega, t, x}})< \infty,
\end{equs}
from which one deduces that for each $i =1,2,3$ and $t \in [0,T]$, $M^i_n(\bar u_n, t)$ are uniformly integrable in $\bar \omega$. 
 Hence, we can pass to the limit in \eqref{eq:martingale-Ml} to obtain, for $i=1,2,3$, 
\begin{equation}       \label{eq:martingale-property}
\bar{\E}V(\bar u|_{[0,s]}, \bar w|_{[0,s]})( M^i(\bar u, t)- M^i(\bar u, s)) =0.
\end{equation}
In addition, using the continuity of $ M^i(\bar u, t)$ in $\phi$, uniform integrability, and the fact that $C^\infty_c(I)$ is dense in $H^{-3}_x$, it follows that  \eqref{eq:martingale-property} holds also for all $\phi \in H^{-3}_x$. Hence, for all $\phi \in H^{-3}_x(I)$, $i = 1,2,3$, one can see that $\bar {M}^i(\bar u , t)$ are continuous $\mathcal{G}_t$-martingales having finite $c$-moments. In particular, by Doob's maximal inequality, they are uniformly integrable (in $t$), which combined with continuity (in $t$) implies that they are also $\bar{\mathcal{F}}_t$-martingales.  By \cite[Prop. A.1]{HOF2} we obtain that  almost surely, for all $\phi \in H^{-3}_x$,  $t \in [0,T]$
\begin{equs}                                                        
\nonumber
(\bar {u}(t), \phi)_{H^{-3}_x}&= (\bar {u}(0), \phi)_{H^{-3}_x}+\int_0^t \big( \Delta (\bar u ^{[m]}(s)), \phi\big)_{H^{-3}_x} \, ds 
\\
&\quad +\sum_{k\in\N}   \int_0^t \big(\sigma(\bar u(s)) e^k, \phi\big)_{H^{-3}_x} \, d \bar w ^k_s.
\label{eq:solving-the-equation}    
\end{equs}
Notice that by \eqref{eq:convergence-initial-conditions}, \eqref{eq:distribution} and \eqref{eq:convergence-in-Z}, it follows that   $\bar{u}(0)\overset{d}{=} u^{(0)} $ and consequently $\bar u (0) \in L^p(\bar \Omega; H^\gamma_x)$. Also, from \eqref{eq:bounds-bar-u} it follows  that $\bar u \in L^{m+1}( \bar\Omega_T ; L^{m+1}_x)$. 
Choosing $\phi = \D^{2}  \psi$  in \eqref{eq:solving-the-equation} for $\psi \in C^\infty_c(I)$, we obtain that for almost all $(\bar {\omega}, t)$
\begin{equs}
\nonumber
(\bar{u}(t), \psi)_{H^{-1}_x}&= (\bar {u}(0), \psi)_{H^{-1}_x }-\int_0^t  \big( \bar u^{[m]}(s), \psi \big)_{L^{2}_x}\,ds
\\
&\quad + \sum_{k\in\N}\int_0^t \big( \sigma(\bar u(s))e^k ,  \psi\big)_{H^{-1}_x} \, d \bar w^k_s.
\end{equs} 
By \cite[Thm. 3.2]{KR_SEE} we have that $\bar u $ is an $ \bar{\mathbb{F}}$-adapted, continuous $H^{-1}_x$-valued process. This shows that $\{(\bar\Omega,\bar \cF,\bar \P), \bar \F, (\bar w^k)_{k\in\N}, \bar u \}$  is a weak solution.

Concerning the claimed bounds:

(i) Estimate \eqref{eq:estimates-space} is obtained in \eqref{eq:bounds-bar-u}.

(ii) For \eqref{eq:down-from-infinity} we have the following. Notice that due to \eqref{eq:Ito}, the quantity  $ \E \|\bar u_n(t)\|_{H^{\gamma}_x}^2$ is differentiable  in $t$, and similarly  to  the argumentation for \eqref{eq:energy1}, one sees that it satisfies
\begin{equs}
\partial_t  \E \|  u_n(t)\|_{H^{\gamma}_x}^2 \leq -  \tfrac{1}{m} \E\| u_n ^{[\frac{m+1}{2}]}(t)\|_{H^{1+\gamma}_x}^2+N,
\end{equs}
where $N$ depends on $\gamma, m, K$ and $T$. 
By the inequalities 
\begin{equs}
\E\| u_n ^{[\frac{m+1}{2}]}(t)\|_{H^{1+\gamma}_x}^2 \geq  \E\| u_n ^{[\frac{m+1}{2}]}(t)\|_{L^2_x}^2 \geq \left( \E \| u_n(t) \|_{L^2_x}^2  \right)^{(m+1)/2} \geq \left( \E \| u_n(t) \|_{H^{\gamma}_x}^2  \right)^{(m+1)/2} ,
\end{equs}
it follows that $g(t):= \E \| u_n(t)\|_{H^{\gamma}_x}^2$ satisfies for almost all $t$ 
\begin{equs}
\partial_t g(t)+ \frac{1}{m} |g(t)|^{(m+1)/2} \leq N. 
\end{equs}
This implies that (see, e.g., \cite[Lemma 5.1]{G13}) with a constant $N$,  depending only on $\gamma, m, K$ and $T$, we have for all $t \in [0,T]$
$$
 \E \|  u_n(t)\|_{H^{\gamma}_x}^2 \leq N t^{-2/(m-1)}
$$
Inequality \eqref{eq:down-from-infinity} follows from the above, again by \eqref{eq:distribution} and lower semicontinuity  of the norms. 

(iii) As before, it suffices to check the bound for $u_n$.
Since $\cX$ embeds continuously into $C_t^{\eps_0}H^{-3}_x$ for some $\eps_0>0$, 
for $\eps\geq\gamma+3$
the statement follows from
\eqref{eq:a-priori-main}, with $\eps'=\eps_0$.
Otherwise let $\theta\in(0,1)$ be the number defined by $\gamma-\eps=(1-\theta)\gamma-3\theta$. Then by interpolation (see Remark \ref{rem:trivial})
\begin{equs}
\|u_n(t)-u_n(s)\|_{H^{\gamma-\eps}_x}
&\leq \|u_n(t)-u_n(s)\|_{H^{\gamma}_x}^{1-\theta}\|u_n(t)-u_n(s)\|_{H^{-3}_x}^{\theta}
\\
&\leq \|u_n\|_{L^\infty_t H^\gamma_x}^{1-\theta}|t-s|^{\theta\eps_0}\|u_n\|_{C^{\eps_0}_tH^{-3}_x}^\theta.
\end{equs}
The statement therefore follows once again from \eqref{eq:a-priori-main} with $\eps'=\theta\eps_0$,  and \eqref{eq:bounds-bar-un}, \eqref{eq:distribution}.
\end{proof}

\section{Strong well-posedness in $H^{-1}$}\label{sec:strong H1}

\begin{proof}[Proof of Proposition \ref{thm:strong-wellposedness}]
In the following we denote $c_0=1/3$, which is $N(-1)$ from Lemma \ref{lem:krylov lemma}, so we have $\delta<\tfrac{2}{c_0}$ and $\bar\delta\leq\tfrac{8m}{c_0(m+1)^2}$.
We verify the assumptions of \cite{KR_SEE}.
Consider the Gelfand triple $L^{m+1}_x\subset H^{-1}_x\equiv (H^{-1}_x)^*\subset (L^{m+1}_x)^*$. 
The inner product in $H^{-1}_x$ as well as the duality between
$L^{m+1}_x$ and $(L^{m+1}_x)^*$ is denoted by $\scal{\cdot,\cdot}$,
so that the two possible interpretations of
$\scal{f,g}$ with $f\in L^{m+1}_x$ and $g\in H^{-1}_x$ agree. 
The operator $A:u\mapsto \Delta u^{[m]}$ maps $L^{m+1}_x$ to $(L^{m+1}_x)^*$ and
$B=(B^k)_{k\in\N}: u\mapsto (\sigma(\cdot,u)e^k)_{k\in\N}$
maps $L^{m+1}_x$ to $\ell^2(H^{-1}_x)$.
We now recall and verify the assumptions from \cite{KR_SEE} in a somewhat more
restrictive form than therein, which will suffice for our purposes.
It is assumed that there exist $\mu>0$, $M\in \R$, such that for all $v,v_1,v_2\in L^{m+1}_x$ the properties $A_1)-A_5)$ below hold:
\begin{itemize}
\item[$A_1)$]\emph{Semicontinuity of $A$: the function $\scal{v,A(v_1+\lambda v_2)}$ is continuous in $\lambda\in \R$.  }

This is a standard fact for the porous medium operator, see  \cite[Ex. 4.1.11]{Rock}. 

\item[$A_2)$]\emph{ Monotonicity of $(A,B)$: }
\begin{equ}
2\scal{v_1-v_2,Av_1-Av_2}+\sum_{k\in \N}\|B^kv_1-B^kv_2\|_{H^{-1}_x}^2
\leq 0.
\end{equ}
First we use Lemma \ref{lem:krylov lemma} to write
\begin{equ}\label{eq:w-p1}
\sum_{k\in \N}\|B^kv_1-B^kv_2\|_{H^{-1}_x}^2=
\sum_{k\in \N}\|\big(\sigma(v_1)-\sigma(v_2)\big)e^k\|_{H^{-1}_x}^2
\leq
c_0\|\sigma(v_1)-\sigma(v_2)\|_{L^2_x}^2.
\end{equ}
Next, observe the elementary inequality, for $f,g:\R\to\R$ satisfying $f'=(g')^2$
\begin{equ}
(a-b)(f(a)-f(b))\geq |g(a)-g(b)|^2,\qquad a,b\in\R.
\end{equ}
This in particular implies
\begin{equ}\label{eq:w-p2}
(a-b)(a^{[m]}-b^{[m]})\geq 
\frac{4m}{(m+1)^2}|a^{[\frac{m+1}{2}]}-b^{[\frac{m+1}{2}]}|^2,\qquad a,b\in\R.
\end{equ}
By \eqref{eq:w-p1}, Assumption \ref{as:sigma} (b), and
\eqref{eq:w-p2} we therefore have
\begin{equs}
2 & \scal{v_1-v_2,Av_1-Av_2}+\sum_{k\in \N}\|B^kv_1-B^kv_2\|_{H^{-1}_x}^2
\\
&\leq
\int_I -2(v_1-v_2)(v_1^{[m]}-v_2^{[m]})+c_0|\sigma(v_1)-\sigma(v_2)|^2\,dx
\\
&\leq
\int_I -2(v_1-v_2)(v_1^{[m]}-v_2^{[m]})+c_0\bar\delta|v_1^{[\frac{m+1}{2}]}-v_2^{[\frac{m+1}{2}]}|^2\,dx
\\
&\leq 0,
\end{equs}
where in the last step we used $\bar\delta\leq \frac{8m}{c_0(m+1)^2}$.
\item [$A_3)$] \emph{Coercivity of $(A,B)$:}
\begin{equ}
2\scal{v,Av}+\sum_{k\in\N}\|B^kv\|_{H^{-1}_x}^2\leq 
-\mu\|v\|_{L^{m+1}_x}^{m+1}+M.
\end{equ}
Using Lemma \ref{lem:krylov lemma} similarly as above, we can write
\begin{equs}
2 \scal{v,Av}+\sum_{k\in\N}\|B^kv\|_{H^{-1}_x}^2
&\leq \int_I \left(-2|v|^{m+1}+c_0|\sigma(v)|^2\right) \,dx
\\
&\leq
(-2+c_0\delta)\|v\|_{L^{m+1}_x}^{m+1}
+2c_0 K.
\end{equs}
We can therefore set $\mu=2-c_0\delta$, which is positive by assumption.

\item [$A_4)$] \emph{Boundedness of the growth of $A$:}
\begin{equ}
\|A v\|_{(L^{m+1}_x)^*}\leq M \|v\|_{L^{m+1}_x}^m.
\end{equ}

This is also standard, see  \cite[Ex.~4.1.11]{Rock}. 

\item [$A_5)$] $\E\|u^{(0)}\|_{H^{-1}_x}^2<\infty$.
This holds by assumption.

\end{itemize}
Invoking \cite[Thms~2.1-2.2]{KR_SEE}, the proof is complete.
\end{proof}

\section*{Acknowledgements}

BG acknowledges financial support by the DFG through the CRC 1283 "Taming uncertainty and profiting from randomness and low regularity in analysis, stochastics and their applications".
MG thanks the support of the Austrian Science Fund (FWF) through the Lise Meitner programme M2250-N32.

\begin{appendices}

\section[Appendix]{Appendix}

\begin{proof}[Proof of Lemma \ref{lem:stroock-varopoulos}]
We assume that $f'=(g')^2$ is bounded, the assumptions of the lemma guarantee that the general case follows from a standard approximation argument.
Denote by $X^{2 \beta}_0 $ the completion of $C_c^\infty(I\times [0, \infty)) $ under the norm 
$$
\|\phi \|^2_{X^{2 \beta}_0}:= C_{\beta} \int_0^\infty\int_I y^{1-2 \beta} |\nabla \phi(x,y)|^2 \, dx  dy ,
$$
where $C_\beta>0$ is a normalizing constant such that (ii) below holds. 
 For $ \psi \in H^{\beta}$ we denote by $E(\psi) \in X^{2\beta}_0$ the unique solution of 
\begin{equs}
\begin{cases} - \nabla \cdot ( y^{1-2\beta} \nabla w)= 0   &\text{in $ I \times (0, \infty)$},
\\
 w =0       & \text{on $ \partial I \times (0, \infty)$},
\\
w=  \psi & \text{on $I\times \{0\}$}.
\end{cases}
\end{equs}
The following facts are well known (see, \cite{Barrios, Pablo} ): 
\begin{enumerate}[(i)]
\item \label{item:trace-inequality}The map $E :H^\beta \to X^{2 \beta}_0$ is an isometry and for all $ \phi \in X^{2\beta}_0$ we have 
\begin{equs}                 
\| \text{Tr} \phi\|_{H^\beta} \leq \| \phi\|_{X^{2 \beta}_0},
\end{equs}
where $\text{Tr}$ is the closure of the operator $\text{Tr}_0$ defined on $C_c^\infty(I\times [0, \infty)) $ by $(\text{Tr}_0\phi)(x)=\phi(x,0)$.

\item \label{item:same-inner-product} For $u \in H^{\beta}$ we have 
\begin{equs}
C_\beta \int_0^\infty \int_I y ^{1-2\beta} \nabla  E(u) \nabla \phi \, dx dy  = \int_I \D^{\beta/2} u \, \D^{\beta/2} \text{Tr}\phi \, dx,
\end{equs}
for all $\phi \in X^{2 \beta}_0$. 
\end{enumerate}
Let $u$  be as in the statement of Lemma \ref{lem:stroock-varopoulos}. Since  $ f (u)= \text{Tr} \, f(E(u))$, we get by applying \eqref{item:same-inner-product} with $\phi = f(E(u))$
\begin{equs}
\int_I f(u) \D^\beta u\,dx &= \int_I\left( \D^{\beta/2}\text{Tr} \, f(E(u))   \right) \left(  \D^{\beta/2} u  \right) \,dx
\\
&=C_\beta \int_0^\infty \int_I y ^{1-2\beta} \nabla f ( E(u)) \nabla E (u) \, dx dy 
\\
& =C_\beta \int_0^\infty \int_I y ^{1-2\beta} f'(E(u))|\nabla  E(u)|^2  \, dx dy.
\label{eq:first-part-varopoulos}
\end{equs}
Using $f'=(g')^2$ and \eqref{item:trace-inequality} we get 
\begin{equs} 
C_\beta \int_0^\infty \int_I y ^{1-2\beta} f'(E(u))|\nabla  E(u)|^2  \, dx dy
& = C_\beta \int_0^\infty \int_I y ^{1-2\beta} |\nabla  g( E(u))|^2  \, dx dy
\\
& \geq \| \text{Tr}\, g(E(u))\|_{H^\beta}^2.  \label{eq:second-part-varopoulos}
\end{equs}
By \eqref{eq:first-part-varopoulos}, \eqref{eq:second-part-varopoulos}, and the equality $\text{Tr}\, g(E(u)) = g(u)$, we get the claim.
\end{proof}

\end{appendices}

\bibliography{PME_spacetime}{}

\begin{thebibliography}{dPQRV11}
\expandafter\ifx\csname url\endcsname\relax
  \def\url#1{\texttt{#1}}\fi
\expandafter\ifx\csname urlprefix\endcsname\relax\def\urlprefix{URL }\fi
\expandafter\ifx\csname href\endcsname\relax
  \def\href#1#2{#2}\fi
\expandafter\ifx\csname burlalt\endcsname\relax
  \def\burlalt#1#2{\href{#2}{\texttt{#1}}}\fi

\bibitem[BCdPS12]{Barrios}
\textsc{B.~Barrios}, \textsc{E.~Colorado}, \textsc{A.~de~Pablo}, and
  \textsc{U.~Sánchez}.
\newblock On some critical problems for the fractional laplacian operator.
\newblock \emph{Journal of Differential Equations} \textbf{252}, no.~11,
  (2012), 6133 -- 6162.
\newblock
  \burlalt{doi:https://doi.org/10.1016/j.jde.2012.02.023}{http://dx.doi.org/https://doi.org/10.1016/j.jde.2012.02.023}.

\bibitem[BCdPS13]{Pablo}
\textsc{C.~Br\"{a}ndle}, \textsc{E.~Colorado}, \textsc{A.~de~Pablo}, and
  \textsc{U.~S\'{a}nchez}.
\newblock A concave-convex elliptic problem involving the fractional
  {L}aplacian.
\newblock \emph{Proc. Roy. Soc. Edinburgh Sect. A} \textbf{143}, no.~1, (2013),
  39--71.
\newblock
  \burlalt{doi:10.1017/S0308210511000175}{http://dx.doi.org/10.1017/S0308210511000175}.

\bibitem[BdPR16]{BDPR}
\textsc{V.~Barbu}, \textsc{G.~da~Prato}, and \textsc{M.~Röckner}.
\newblock \emph{Stochastic Porous Media Equations}.
\newblock Springer International Publishing, 2016.
\newblock
  \burlalt{doi:10.1007/978-3-319-41069-2}{http://dx.doi.org/10.1007/978-3-319-41069-2}.

\bibitem[BFRO17]{Figalli}
\textsc{M.~Bonforte}, \textsc{A.~Figalli}, and \textsc{X.~Ros-Oton}.
\newblock Infinite speed of propagation and regularity of solutions to the
  fractional porous medium equation in general domains.
\newblock \emph{Communications on Pure and Applied Mathematics} \textbf{70},
  no.~8, (2017), 1472--1508.
\newblock \burlalt{doi:10.1002/cpa.21673}{http://dx.doi.org/10.1002/cpa.21673}.

\bibitem[BM19]{bailleul2019}
\textsc{I.~Bailleul} and \textsc{A.~Mouzard}.
\newblock Paracontrolled calculus for quasilinear singular pdes (2019).
\newblock \burlalt{arXiv:1912.09073}{http://arxiv.org/abs/1912.09073}.

\bibitem[BV15]{Bonforte-Vazquez}
\textsc{M.~Bonforte} and \textsc{J.~L. V{\'{a}}zquez}.
\newblock A priori estimates for fractional nonlinear degenerate diffusion
  equations on bounded domains.
\newblock \emph{Archive for Rational Mechanics and Analysis} \textbf{218},
  no.~1, (2015), 317--362.
\newblock
  \burlalt{doi:10.1007/s00205-015-0861-2}{http://dx.doi.org/10.1007/s00205-015-0861-2}.

\bibitem[BVW15]{BVW15}
\textsc{C.~Bauzet}, \textsc{G.~Vallet}, and \textsc{P.~Wittbold}.
\newblock A degenerate parabolic-hyperbolic {{Cauchy}} problem with a
  stochastic force.
\newblock \emph{J. Hyperbolic Differ. Equ.} \textbf{12}, no.~3, (2015),
  501--533.

\bibitem[CS07]{C-S}
\textsc{L.~Caffarelli} and \textsc{L.~Silvestre}.
\newblock {An Extension Problem Related to the Fractional Laplacian}.
\newblock \emph{Communications in Partial Differential Equations} \textbf{32},
  no.~8, (2007), 1245--1260.
\newblock
  \burlalt{doi:10.1080/03605300600987306}{http://dx.doi.org/10.1080/03605300600987306}.

\bibitem[DGG19]{DGG19}
\textsc{K.~Dareiotis}, \textsc{M.~Gerencs{\'e}r}, and \textsc{B.~Gess}.
\newblock Entropy solutions for stochastic porous media equations.
\newblock \emph{Journal of Differential Equations} \textbf{266}, no.~6, (2019),
  3732--3763.
\newblock \burlalt{arXiv:1803.06953}{http://arxiv.org/abs/1803.06953}.
\newblock
  \burlalt{doi:10.1016/j.jde.2018.09.012}{http://dx.doi.org/10.1016/j.jde.2018.09.012}.

\bibitem[DGT19]{DGT19}
\textsc{K.~Dareiotis}, \textsc{B.~Gess}, and \textsc{P.~Tsatsoulis}.
\newblock Ergodicity for {{Stochastic Porous Media Equations}}.
\newblock \emph{arXiv:1907.04605 [math]} (2019).
\newblock \burlalt{arXiv:1907.04605}{http://arxiv.org/abs/1907.04605}.

\bibitem[DHV16]{Debussche}
\textsc{A.~Debussche}, \textsc{M.~Hofmanov\'{a}}, and \textsc{J.~Vovelle}.
\newblock Degenerate parabolic stochastic partial differential equations:
  quasilinear case.
\newblock \emph{Ann. Probab.} \textbf{44}, no.~3, (2016), 1916--1955.
\newblock
  \burlalt{doi:10.1214/15-AOP1013}{http://dx.doi.org/10.1214/15-AOP1013}.

\bibitem[dPQRV11]{Vazquez_fractional}
\textsc{A.~de~Pablo}, \textsc{F.~Quir{\'o}s}, \textsc{A.~Rodr{\'i}guez}, and
  \textsc{J.~L. V{\'a}zquez}.
\newblock A fractional porous medium equation.
\newblock \emph{Advances in Mathematics} \textbf{226}, no.~2, (2011), 1378 --
  1409.
\newblock
  \burlalt{doi:https://doi.org/10.1016/j.aim.2010.07.017}{http://dx.doi.org/https://doi.org/10.1016/j.aim.2010.07.017}.

\bibitem[FG95]{Flandoli-Gatarek}
\textsc{F.~Flandoli} and \textsc{D.~Gatarek}.
\newblock {Martingale and stationary solutions for stochastic Navier-Stokes
  equations}.
\newblock \emph{Probability Theory and Related Fields} \textbf{102}, no.~3,
  (1995), 367--391.
\newblock
  \burlalt{doi:10.1007/BF01192467}{http://dx.doi.org/10.1007/BF01192467}.

\bibitem[FG18]{FG18}
\textsc{B.~Fehrman} and \textsc{B.~Gess}.
\newblock Path-by-path well-posedness of nonlinear diffusion equations with
  multiplicative noise.
\newblock \emph{arXiv preprint arXiv:1807.04230} (2018).

\bibitem[Ger20]{GERENCSER2020}
\textsc{M.~Gerencsér}.
\newblock Nondivergence form quasilinear heat equations driven by space-time
  white noise.
\newblock \emph{Annales de l'Institut Henri Poincaré C, Analyse non linéaire}
  (2020).
\newblock
  \burlalt{doi:https://doi.org/10.1016/j.anihpc.2020.01.003}{http://dx.doi.org/https://doi.org/10.1016/j.anihpc.2020.01.003}.

\bibitem[Ges13]{G13}
\textsc{B.~Gess}.
\newblock Random attractors for degenerate stochastic partial differential
  equations.
\newblock \emph{J. Dynam. Differential Equations} \textbf{25}, no.~1, (2013),
  121--157.
\newblock
  \burlalt{doi:10.1007/s10884-013-9294-5}{http://dx.doi.org/10.1007/s10884-013-9294-5}.

\bibitem[GH18]{GH18}
\textsc{B.~Gess} and \textsc{M.~Hofmanov{\'a}}.
\newblock Well-posedness and regularity for quasilinear degenerate
  parabolic-hyperbolic {{SPDE}}.
\newblock \emph{The Annals of Probability} \textbf{46}, no.~5, (2018),
  2495--2544.
\newblock
  \burlalt{doi:10.1214/17-AOP1231}{http://dx.doi.org/10.1214/17-AOP1231}.

\bibitem[GH19]{GerencserHairer}
\textsc{M.~Gerencs\'{e}r} and \textsc{M.~Hairer}.
\newblock A solution theory for quasilinear singular {SPDE}s.
\newblock \emph{Comm. Pure Appl. Math.} \textbf{72}, no.~9, (2019), 1983--2005.
\newblock \burlalt{doi:10.1002/cpa.21816}{http://dx.doi.org/10.1002/cpa.21816}.

\bibitem[GIP15]{GIP}
\textsc{M.~Gubinelli}, \textsc{P.~Imkeller}, and \textsc{N.~Perkowski}.
\newblock {Paracontrolled distributions and singular PDEs}.
\newblock \emph{Forum of Mathematics, Pi} \textbf{3}, (2015), e6.
\newblock
  \burlalt{doi:10.1017/fmp.2015.2}{http://dx.doi.org/10.1017/fmp.2015.2}.

\bibitem[GRZ09]{RocknerGoldys}
\textsc{B.~Goldys}, \textsc{M.~R\"{o}ckner}, and \textsc{X.~Zhang}.
\newblock Martingale solutions and {M}arkov selections for stochastic partial
  differential equations.
\newblock \emph{Stochastic Process. Appl.} \textbf{119}, no.~5, (2009),
  1725--1764.
\newblock
  \burlalt{doi:10.1016/j.spa.2008.08.009}{http://dx.doi.org/10.1016/j.spa.2008.08.009}.

\bibitem[GS17]{GS16-2}
\textsc{B.~Gess} and \textsc{P.~E. Souganidis}.
\newblock Stochastic non-isotropic degenerate parabolic--hyperbolic equations.
\newblock \emph{Stochastic Process. Appl.} \textbf{127}, no.~9, (2017),
  2961--3004.

\bibitem[Hai14]{H0}
\textsc{M.~Hairer}.
\newblock A theory of regularity structures.
\newblock \emph{Inventiones mathematicae} \textbf{198}, no.~2, (2014),
  269--504.
\newblock \burlalt{arXiv:1303.5113}{http://arxiv.org/abs/1303.5113}.
\newblock
  \burlalt{doi:10.1007/s00222-014-0505-4}{http://dx.doi.org/10.1007/s00222-014-0505-4}.

\bibitem[Hof13]{HOF2}
\textsc{M.~Hofmanov\'{a}}.
\newblock Degenerate parabolic stochastic partial differential equations.
\newblock \emph{Stochastic Process. Appl.} \textbf{123}, no.~12, (2013),
  4294--4336.
\newblock
  \burlalt{doi:10.1016/j.spa.2013.06.015}{http://dx.doi.org/10.1016/j.spa.2013.06.015}.

\bibitem[KR81]{KR_SEE}
\textsc{N.~V. Krylov} and \textsc{B.~L. Rozovskii}.
\newblock Stochastic evolution equations.
\newblock \emph{Journal of Soviet Mathematics} \textbf{16}, no.~4, (1981),
  1233--1277.
\newblock
  \burlalt{doi:10.1007/BF01084893}{http://dx.doi.org/10.1007/BF01084893}.

\bibitem[Kry99]{K_Lp}
\textsc{N.~V. Krylov}.
\newblock An analytic approach to {SPDE}s.
\newblock In \emph{Mathematical Surveys and Monographs},  185--242. American
  Mathematical Society ({AMS}), 1999.
\newblock
  \burlalt{doi:10.1090/surv/064/05}{http://dx.doi.org/10.1090/surv/064/05}.

\bibitem[KS91]{Karatzas}
\textsc{I.~Karatzas} and \textsc{S.~E. Shreve}.
\newblock \emph{Brownian motion and stochastic calculus}, vol. 113 of
  \emph{Graduate Texts in Mathematics}.
\newblock Springer-Verlag, New York, second ed., 1991.
\newblock
  \burlalt{doi:10.1007/978-1-4612-0949-2}{http://dx.doi.org/10.1007/978-1-4612-0949-2}.

\bibitem[LM01]{LMG01}
\textsc{P.-L. Lions} and \textsc{S.~{Mas-Gallic}}.
\newblock Une m{\'e}thode particulaire d{\'e}terministe pour des {\'e}quations
  diffusives non lin{\'e}aires.
\newblock \emph{Comptes Rendus de l'Acad{\'e}mie des Sciences. S{\'e}rie I.
  Math{\'e}matique} \textbf{332}, no.~4, (2001), 369--376.
\newblock
  \burlalt{doi:10.1016/S0764-4442(00)01795-X}{http://dx.doi.org/10.1016/S0764-4442(00)01795-X}.

\bibitem[LR15]{Liu}
\textsc{W.~Liu} and \textsc{M.~R\"{o}ckner}.
\newblock \emph{Stochastic partial differential equations: an introduction}.
\newblock Universitext. Springer, Cham, 2015.
\newblock
  \burlalt{doi:10.1007/978-3-319-22354-4}{http://dx.doi.org/10.1007/978-3-319-22354-4}.

\bibitem[MMP14]{Mueller-M-P}
\textsc{C.~Mueller}, \textsc{L.~Mytnik}, and \textsc{E.~Perkins}.
\newblock Nonuniqueness for a parabolic spde with $\frac{3}{4}-\varepsilon
  $-hölder diffusion coefficients.
\newblock \emph{Ann. Probab.} \textbf{42}, no.~5, (2014), 2032--2112.
\newblock \burlalt{doi:10.1214/13-AOP870}{http://dx.doi.org/10.1214/13-AOP870}.

\bibitem[MR93]{Meleard}
\textsc{S.~M\'{e}l\'{e}ard} and \textsc{S.~Roelly}.
\newblock Interacting measure branching processes. {S}ome bounds for the
  support.
\newblock \emph{Stochastics Stochastics Rep.} \textbf{44}, no. 1-2, (1993),
  103--121.
\newblock
  \burlalt{doi:10.1080/17442509308833843}{http://dx.doi.org/10.1080/17442509308833843}.

\bibitem[OW19]{OttoWeber}
\textsc{F.~Otto} and \textsc{H.~Weber}.
\newblock Quasilinear {SPDE}s via rough paths.
\newblock \emph{Arch. Ration. Mech. Anal.} \textbf{232}, no.~2, (2019),
  873--950.
\newblock
  \burlalt{doi:10.1007/s00205-018-01335-8}{http://dx.doi.org/10.1007/s00205-018-01335-8}.

\bibitem[Par75]{P75}
\textsc{{\'E}.~Pardoux}.
\newblock Equations aux d{\'e}riv{\'e}es partielles stochastiques non
  lin{\'e}aires monotones.
\newblock \emph{PhD thesis} (1975).

\bibitem[PR07]{Rock}
\textsc{C.~Pr\'{e}v\^{o}t} and \textsc{M.~R\"{o}ckner}.
\newblock \emph{A concise course on stochastic partial differential equations},
  vol. 1905 of \emph{Lecture Notes in Mathematics}.
\newblock Springer, Berlin, 2007.

\bibitem[Tri78]{Triebel}
\textsc{H.~Triebel}.
\newblock \emph{Interpolation theory, function spaces, differential operators},
  vol.~18 of \emph{North-Holland Mathematical Library}.
\newblock North-Holland Publishing Co., Amsterdam-New York, 1978.

\end{thebibliography}
\bibliographystyle{Martin}

\end{document}